\documentclass[11pt]{article}
\usepackage{latexsym,a4wide}
\usepackage{amssymb}
\usepackage{amsmath}
\usepackage{color}
\usepackage{graphicx}
\usepackage{amsthm}
\usepackage{indentfirst}
\allowdisplaybreaks

\newcommand \bel {\begin{equation}\label}
\newcommand \ee {\end{equation}}
\newcommand \be {\begin{equation}}

\newcommand \RR {\mathbb R}
\newcommand \HH {\mathbb H}

\newcommand \LL {\mathbb L}

\newcommand \del \partial

\newcommand \Lcal {\mathcal L}
\newcommand \Dcal {\mathcal D}
\newcommand \Tcal {\mathcal T}
\newcommand \bei {\begin{itemize}}
\newcommand \eei {\end{itemize}}

\newtheorem{theorem}{\color{black}\indent Theorem}[section]
\newtheorem{lemma}{\color{black}\indent Lemma}[section]
\newtheorem{proposition}{\color{black}\indent Proposition}[section]

\newtheorem{remark}{\color{black}\indent Remark}[section]

\begin{document}
\large
\title{Dynamical behavior near self-similar blowup waves for the generalized $b$-equation}
\author{
{\sc Weiping Yan}\thanks{School of Mathematics, Xiamen University, Xiamen 361000, P.R. China. Email: yanwp@xmu.edu.cn.}
\thanks{Laboratoire Jacques-Louis Lions, Sorbonne Université, 4, Place Jussieu, 75252 Paris, France.}
}
\date{June 17, 2018}

\maketitle

\begin{abstract}  
In this paper, we consider the explicit wave-breaking mechanism and its dynamical behavior near this singularity for the generalized $b$-equation. This generalized $b$-equation arises from the shallow water theory, which includes the Camassa-Holm equation, the Degasperis-Procesi equation, the Fornberg-Whitham equation, the Korteweg-de Vires equation and the classical $b$-equation. More precisely, we find that there exists an explicit self-similar blowup solution for the generalized $b$-equation. Meanwhile, this self-similar blowup solution is asymptotic stability in a parameters domain, but instability in other parameters domain.
\end{abstract}



\section{Introduction and main results} 
\setcounter{equation}{0}

The $b$-equation arises from nonlinear shallow water theory, which takes the following form
\bel{E1-1R0}
\aligned
&m_t+c_0u_x+um_x+bu_xm+\Gamma u_{xxx}=0,\quad\quad t>0,\quad x\in\RR,\\
&u(0,x)=u_0(x),
\endaligned
\ee
where $c_0$, $b$, $\Gamma$ are arbitrary real constants, the notation $m=u-\alpha^2u_{xx}$ with constant $\alpha\in\RR$. 

Equation (\ref{E1-1R0}) can be derived as the family of asymptotically equivalent shallow water wave equations that emerges at quadratic order accuracy for any $b\neq-1$
by an appropriate Kodama transformation \cite{DGH1,DGH2}. The $b$-equation can be written as a nonlinear dispersive equation
$$
\aligned
&u_t-\alpha^2u_{txx}+c_0u_x+(b+1)uu_x+\Gamma u_{xxx}=\alpha^2(bu_xu_{xx}+uu_{xxx}),\quad\quad t>0,\quad x\in\RR,\\
&u(0,x)=u_0(x), 
\endaligned
$$
which contains the three kinds of famous shallow water equation. One can see \cite{EY} for the well-posedness, blow-up phenomena, and global solutions for $b$-equation (\ref{E1-1R0}).

In this paper we study a generalized $b$-equation as follows
\bel{E1-1R00}
\aligned
&u_t-\alpha^2u_{txx}+c_0u_x+(b+1)uu_x+\Gamma u_{xxx}=\alpha^2(c_1u_xu_{xx}+c_2uu_{xxx}),\quad\quad t>0,\quad x\in\RR,\\
&u(0,x)=u_0(x), 
\endaligned
\ee
where we introduce two more parameters $c_1,c_2\in\RR$ than equation (\ref{E1-1R0}) to contain more shallow water equations, that is, the Camassa-Holm (CH) equation, the Degasperis-Procesi (DP) equation, the Fornberg-Whitham (FW) equation, and the Korteweg-de Vires (KdV) equation.

If $\alpha=1$, $c_0=\Gamma=0$, $b=2$, $c_1=2$ and $c_2=1$, then equation (\ref{E1-1R00}) becomes the famous Camassa-Holm equation \cite{CH}
$$
u_t-u_{txx}+3uu_x=2u_xu_{xx}+uu_{xxx},
$$
which admits a bi-Hamiltonian structure \cite{CH,C11}, and it can be written as 
$$
U_t=-J_1{\del H_2\over \del U}=-J_2{\del H_1\over \del U},
$$
where
$$
\aligned
&U=u-u_{xx},\quad J_1=\del-\del^3,\quad J_2=\del U+U\del,\\
& H_1={1\over 2}\int_{\RR}(u^2+u_x^2)dx,\quad H_2={1\over 2}\int_{\RR}(u^2+uu_x^2)dx.
\endaligned
$$
Moreover, it has an exact peaked soliton
$$
u_c(t,x)=ce^{-|x-ct|},
$$
where constant $c>0$ is the speed of wave. The stability of it has been obtained in \cite{C3,C4,C5,C6,HK}.

The Camassa-Holm equation is one of most famous shallow water models, which has attracted hundreds of papers to study it and its generalizations \cite{Bre1,Bre2,C9,C11,C7,C1,C10,C8,HK}.
One of important topic focusing on the CH equation is the wave break phenomenon. A sufficient condition for the breakdown criterion depends on the value of initial data at one point $x_0\in\RR$. 
The classical result is obtained by Constantin and Escher \cite{C2}, this criterion required 
the initial data $u_0(x)\in\HH^3(\RR)$ at some point $x_0\in\RR$ such that $u'(x_0)<-{\sqrt{2}\over 2}\|u_0\|_{\HH^1}$.
After this result, there are many results concerning on the improvement of initial data condition at some point $x_0\in\RR$, one can see \cite{C7,C0,Liu,Mc} for more details. 
Recently, Brandolese \cite{Br1} gave a new sufficient condition for the breakdown, it needs the initial data satisfies $u_0'(x_0)+|u_0(x_0)|<0$ for at least one point $x_0\in\RR$.

If $\alpha=1$, $c_0=\Gamma=0$, $b=3=c_1=3$ and $c_2=1$, then equation (\ref{E1-1R00}) becomes the famous Degasperis-Procesi equation
$$
u_t-u_{txx}+4uu_x=3u_xu_{xx}+uu_{xxx}, \quad (t,x)\in\RR^+\times\RR.
$$
which was introduced by Degasperis-Procesi \cite{DP} also to model the propagation of unidirectional shallow water waves over a flat bottom. It is a geodesic flow of a rigid invariant symmetric linear connection on the diffeomorphism group of the circle \cite{EK}.

The DP equation can be rewritten in Hamiltonian form as follows
$$
m_t=B_1{\del H_1\over \del m}=B_2{\del H_2\over \del m},
$$
where $B_1$ and $B_2$ form a compatible bi-Hamiltonian pair \cite{Ho}, and 
$$
m=u-u_{xx},\quad  B_1=\del_x(1-\del_x^2)(4-\del_x^2),\quad B_2=m^{{2\over 3}}\del_xm^{1\over 3}(\del_x-\del_x^3)m^{{1\over 3}}\del_xm^{2\over 3}.
$$
It admits the conservation laws
$$
E_1(u)=\int_{\RR}(1-\del_x^2)udx,\quad E_2(u)=\int_{\RR}\Big((1-\del_x^2)u\Big)\Big((4-\del_x^2)^{-1}u\Big)dx,\quad E_3(u)=\int_{\RR}u^3dx,
$$
the DP equation has not only an exact peakon \cite{DHH}, but also shock peakons \cite{Lun}
$$
u(t,x)=-{1\over t+k}sgn(x)e^{-|x|},\quad k>0.
$$

The CH equation and DP equation can rewritten as the following coupled system
$$
\aligned
&m=u-u_{xx},\\
&m_t+m_xu+\lambda mu_x=0,
\endaligned
$$
where it is the CH equation for $\lambda=2$, it turns to the DP equation for $\lambda=3$.
Both two equations possess smooth solutions that develop singularities in finite time via a process that captures the essential features of breaking waves \cite{W}. There are many papers to study the well-posedness theory and blowup analysis for the DP equation and blowup analysis \cite{CH,ELY,HC,HHG,LY,Yin1,Yin2}.
One of the important features of the DP equation is that it not only has peakon solitons \cite{DHH}, but also shock peakons \cite{Lun}.

If $\alpha=1$, $c_0=-1$, $\Gamma=0$, $b={1\over2}$, $c_1={9\over2}$ and $c_2={3\over2}$, then equation (\ref{E1-1R00}) becomes
the Fornberg-Whitham (FW) equation, which was first introduced by Whitham \cite{W} in 1967, Whitham and Fornberg \cite{FW} in 1978 to model the shallow water, it has the form 
\bel{E1-1R1}
u_t-u_{txx}-u_x+{3\over 2}uu_x-{3\over 2}uu_{xxx}-{9\over 2}u_xu_{xx}=0, \quad (t,x)\in\RR^+\times\RR,
\ee
with the initial data
$$
u(0,x)=u_0(x),
$$
where $\omega={1\over c_0}=\sqrt{gh}>0$ is a positive constant, $h$ and $g$ are the mean fluid depth and the gravitational constant, respectively.
The FW equation is not integrable and there is no useful conservation laws which would be used to make estimates of solutions, meanwhile, it does not only allow traveling wave solutions like the KdV equation, but also possess peakon solutions as some Camassa-Holm type equations. 
On can see \cite{FW, Iv, Wei} for more details.

If $\alpha=c_0=0$, $b=2$ and $\Gamma=1$, then equation (\ref{E1-1R00}) becomes the famous KdV equation
$$
u_t-3uu_x+u_{xxx}=0,
$$
which describes the unidirectional propagation of waves at the free surface of shallow water under the influence of gravity.
It is well-know that the KdV equation with $\omega=0$ admits a smooth soliton \cite{DJ}
$$
u_c(t,x)=a_0\sec h^2(\mu_0(x-ct)),
$$
where parameters $\mu_0={\sqrt{a_0}\over 2}$ and $c=c_0+a_0$.
Bourgain \cite{Bo} showed that solutions to the KdV equation are global as long as the initial data is square integrable.
One can also see related result in \cite{KPV,Tao}.

Recently, Yan \cite{Yan1,Yan2} found there are explicit self-similar blowup solutions for a class of shallow water equations, including the Camassa-Holm equation, the Degasperis-Procesi equations, the Dullin-Gottwald-Holm equation, the Korteweg-de Vires equation, the dispersive rod equation and the Benjamin-Bona-Mahony equation. Meanwhile, those explicit self-similar blowup solutions for the Camassa-Holm equation, the Degasperis-Procesi equations, the Dullin-Gottwald-Holm equation and the dispersive rod equation are asymptotic stability, but for the Korteweg-de Vires equation and the Benjamin-Bona-Mahony equation are instability.
In the present paper, we study the explicit wave breaking phenomenon for $b$-equation (\ref{E1-1R00}). More precisely, we first give an explicit self-similar blowup solution of $b$-equation (\ref{E1-1R00}), then asymptotic stability and instability of this blowup solution are shown according to the domain of parameters $c_1,c_2,b$.

Here the first result gives the existence of explicit self-similar solution.

\begin{theorem}
The generalized $b$-equation (\ref{E1-1R00}) admits an explicit self-similar blowup solution
\bel{ER1-01}
u_0(t,x)=-{1\over b+1}({x\over T-t}+c_0),
\ee
where constant $b\neq-1$.
\end{theorem}

The second result gives the asympototic stability of self-similar solution given in (\ref{ER1-01}) by parameters assumptions.

\begin{theorem}
Let $s>3$.
Assume that parameters $\alpha\neq0$ and $b,c_1,c_2$ satisfy
$$
\aligned
&{2(b+1)+c_2-2c_1\over b+1}>0,\\
&{2c_1+1-c_2\over b+1}>0.
\endaligned
$$
Then the explicit self-similar blowup solution (\ref{ER1-01})
of generalized $b$-equation (\ref{E1-1R00}) is asymptotic stability, that is, for a sufficient small $\sigma$, if 
$$
\|u_0(x)+{1\over b+1}({x\over T}+c_0)\|_{\HH^s}\leq\sigma,
$$
then there is a solution $u(t,x)$ of equation (\ref{E1-1R00}) such that
$$
\|u(t,x)-u_0(t,x)\|_{\HH^s}\lesssim (T-t)^{3(b+1)+c_2-2c_1\over b+1},\quad \forall (t,x)\in(0,T)\times\RR.
$$
\end{theorem}

The last result states the instability of self-similar blowup solution given in (\ref{ER1-01}) by parameters assumptions.

\begin{theorem}
Let $s>2$.
Assume that parameters $\alpha\neq0$ and $c_1,c_2,b$ satisfy
$$
{2(b+1)+c_2-2c_1\over b+1}<0,
$$
or
$$
\aligned
&2(b+1)+c_2-2c_1=0,\\
&{2c_1+1-c_2\over b+1}< 0.
\endaligned
$$
Then the explicit self-similar blowup solution (\ref{ER1-01})
of generalized $b$-equation (\ref{E1-1R00}) is instability, that is, no matter how small $\sigma>0$, if
$$
\|u_0(x)+{1\over b+1}({x\over T}+c_0)\|_{\HH^s}\leq\sigma,
$$
then any solution $u(t,x)\in\HH^s$ of equation (\ref{E1-1R00}) such that
$$
\|u(t,x)-u_0(t,x)\|_{\HH^s}> {1\over (T-t)^{C_{c_1,c_2,b}}},\quad \forall (t,x)\in(0,T)\times\RR.
$$
where $C_{c_1,c_2,b}$ is a positive constant depending on $c_1,c_2,b$.
\end{theorem}

\begin{remark}
It follows from (\ref{ER1-01}) that there is type I singularity of a class of shallow water equations including the Camassa-Holm equation, the Degasperis-Procesi equation, the Fornberg-Whitham equation, and the Korteweg-de Vires equation. There is 
$$
\del_xu_0(t,x)|_{x=0}=-{1\over(b+1)(T-t)}\rightarrow\infty,\quad\quad t\rightarrow T^{-}.
$$

Moreover, if parameters $\alpha\neq0$ and $b,c_1,c_2$ satisfy
$$
\aligned
&{2(b+1)+c_2-2c_1\over b+1}>0,\\
&{2c_1+1-c_2\over b+1}>0,
\endaligned
$$
then the generalized $b$-equation (\ref{E1-1R00}) has a stable self-similar blowup solution in Sobolev space $\HH^s$ with $s>3$. This includes the Camassa-Holm equation and the Degasperis-Procesi equation. One can see \cite{Yan2} for the corresponding results on those two famous equations. 

If parameters $\alpha\neq0$ and $c_1,c_2,b$ satisfy
$$
{2(b+1)+c_2-2c_1\over b+1}<0,
$$
or
$$
\aligned
&2(b+1)+c_2-2c_1=0,\\
&{2c_1+1-c_2\over b+1}< 0.
\endaligned
$$
then the generalized $b$-equation (\ref{E1-1R00}) has a unstable self-similar blowup solution in Sobolev space $\HH^s$ with $s>2$. This includes the Fornberg-Whitham equation (\ref{E1-1R1}). Here we point out the completely classification of stability and instability of solution (\ref{ER1-01}) according to the parameters is still open.
\end{remark}

Thoughout this paper, we denote the usual norm of $\mathbb{L}^2(\RR)$ and $\HH^{s}(\RR)$ by $\|\cdot\|_{\mathbb{L}^2}$and $\|\cdot\|_{\HH^{s}}$, respectively. $\star$ stands for the convolution. $[A,B]$ denotes the commutator of two linear operators $A$ and $B$.  
$\Dcal(\Lcal)$ is the domain of the operator $\Lcal$. The symbol $a\lesssim b$ means that there exists a positive constant $C$ such that $a\leq Cb$.

The organization of this paper is as follows. In section 2, we give the details of finding explicit self-similar solution of the generalized $b$-equation (\ref{E1-1R00}). This idea follows from \cite{Yan1,Yan2}.
The last section is to study the dynamical behavior near those explicit self-similar solutions, that is, giving the proof of Theorem 1.2 and Theorem 1.3.


\section{The explicit self-similar solutions}\setcounter{equation}{0}

Let parameter $T$ be a positive constant. Introduce the similarity coordinates
\bel{RE2-2}
\tau=-\log(T-t), \quad \rho=\frac{x}{T-t},
\ee
then we denote by
$$
u(t,x)=\phi(-\log(T-t),\frac{x}{T-t}),
$$

Thus $b$-equation (\ref{E1-1R00}) is transformed into an one dimensional quasilinear equation
\bel{E2-2}
\phi_{\tau}+\Big(\rho+c_0+(b+1)\phi\Big)\phi_{\rho}-\alpha^2e^{2\tau}(\phi_{\tau\rho\rho}+2\phi_{\rho\rho})
+(\Gamma-\alpha^2\rho)e^{2\tau}\phi_{\rho\rho\rho}=\alpha^2e^{2\tau}(c_1\phi_{\rho}\phi_{\rho\rho}+c_2\phi\phi_{\rho\rho\rho}).
\ee


Consequently, it is easy to check that the generalized $b$-equation (\ref{E1-1R00}) has an explicit self-similar solution
\bel{E2-2-02}
u(t,x)=-{1\over b+1}({x\over T-t}+c_0).
\ee
Above process gives the proof of Theorem 1.1.

Furthermore, we compare those self-similiar blowup solutions between four classical shallow water models.

Let $\alpha=1$, $c_0=\Gamma=0$, $b=2$, $c_1=2$ and $c_2=1$, we get the Camassa-Holm equation. Using (\ref{E2-2-02}), the CH equation admits a self-similar solution blowup solution as follows
$$
u(t,x)=-{x\over 3(T-t)}.
$$

Let $\alpha=1$, $c_0=\Gamma=0$, $b=3=c_1=3$ and $c_2=1$,  the generalized $b$-equation (\ref{E1-1R00}) is reduced into the Degasperis-Procesi equation, and it has a self-similar solution blowup solution as follows
$$
u(t,x)=-{x\over 4(T-t)}.
$$

Let $\alpha=1$, $c_0=-1$, $\Gamma=0$, $b={1\over2}$, $c_1={9\over2}$ and $c_2={3\over2}$, the generalized $b$-equation (\ref{E1-1R00}) becomes
the Fornberg-Whitham equation which has a self-similar solution blowup solution as follows
$$
u(t,x)=-{2\over3}({x\over T-t}-1).
$$

Let $\alpha=c_0=0$, $b=2$ and $\Gamma=1$, then equation (\ref{E1-1R00}) becomes the famous KdV equation, which has a self-similar solution blowup solution as follows
$$
u(t,x)=-{x\over 3(T-t)}.
$$

In conclusion, one can see that the generalized $b$-equation, the Camassa-Holm equation, the Degasperis-Procesi equation, the Fornberg-Whitham equation and the Korteweg-de Vires equation can have a common blowup profile ${x\over T-t}$, so they have the
same wave breaking behavior, that is
$$
\partial_{x}u|_{x=0}\rightarrow+\infty,~~as~~t\rightarrow T^{-}.
$$

\


\section{Dynamical behavior near self-similar blowup solutions } 
In this section, we consider nonlinear stability and instability of explicit self-similar blowup solution given in section 2
 for the generalized $b$-equation (\ref{E1-1R00}) according to the domain of parameters.

Let the solution of generalized $b$-equation is the form 
\bel{E4-1-0}
u(t,x)=v(t,x)+u_0(t,x),
\ee
where 
$$
u_0(t,x)=-{1\over b+1}({x\over T-t}+c_0)
$$ 
is the explicit self-similar solution for equation (\ref{E1-1R00}). 

Substituting (\ref{E4-1-0}) into (\ref{E1-1R00}), we get an equation in the similarity coordinates (\ref{RE2-2}) as follows
\bel{E4-2}
\aligned
v_{\tau}-\alpha^2e^{2\tau}v_{\tau\rho\rho}-\alpha^2e^{2\tau}(2-{c_1\over b+1})v_{\rho\rho}+&e^{2\tau}\Big(\Gamma+{\alpha^2c_0c_2\over b+1}+{\alpha^2(c_2-b-1)\over b+1}\rho\Big)v_{\rho\rho\rho}\\
&-v+(b+1)vv_{\rho}=\alpha^2e^{2\tau}(c_1v_{\rho}v_{\rho\rho}+c_2vv_{\rho\rho\rho}).
\endaligned
\ee

Introduce the transformation 
$$
\overline{v}(\tau,\rho_0)=e^{-\tau}v(\tau,\rho)
$$
where 
$$
\rho_0:=e^{-\tau}\rho.
$$
Then equation (\ref{E4-2}) becomes
\bel{E4-2r0}
\aligned
\overline{v}_{\tau}-\alpha^2\overline{v}_{\tau\rho_0\rho_0}-\alpha^2(1-{c_1\over b+1})\overline{v}_{\rho_0\rho_0}+e^{-\tau}\Big(\Gamma+{\alpha^2c_0c_2\over b+1}&+{\alpha^2c_2\over b+1}e^{\tau}\rho\Big)\overline{v}_{\rho_0\rho_0\rho_0}
+\Big((b+1)\overline{v}-\rho_0\Big)\overline{v}_{\rho_0}\\
&=\alpha^2(c_1\overline{v}_{\rho_0}\overline{v}_{\rho_0\rho_0}+c_2\overline{v}~\overline{v}_{\rho_0\rho_0\rho_0}).
\endaligned
\ee

Note that the operator $1-\alpha^2\del_{\rho_0\rho_0}$ has a fundamental solution 
$$
p(x)={1\over 2\alpha}e^{-|{\rho_0\over \alpha}|}.
$$

We denote the operator $(1-\alpha^2\del_{\rho_0\rho_0})^{{1\over 2}}$ by $\Lambda$, then $\Lambda^{-2}\overline{v}=p(\rho_0)\star \overline{v}$ for all $\overline{v}\in\LL^2$.
Let 
$$
w(\tau,\rho_0)=\overline{v}(\tau,\rho_0)-\alpha^2\overline{v}_{\rho_0\rho_0}(\tau,\rho_0),
$$ 
then $v(\tau,\rho_0)=p\star w$, where $\rho_0\in\RR$ and  $\star$ denotes the convolution.
Thus equation (\ref{E4-2r0}) can be rewritten as a non-local equation 
\bel{E4-3}
\aligned
w_{\tau}+(1-{c_1\over b+1})w-&e^{-\tau}\Big({\Gamma\over\alpha^2}+{c_2(c_0+e^{\tau}\rho_0)\over b+1}\Big)w_{\rho_0}-(1-{c_1\over b+1})(p\star w)\\
&\quad\quad+e^{-\tau}\Big({\Gamma\over\alpha^2}+{c_0c_2+(c_2-b-1)e^{\tau}\rho)\over b+1}\Big)(p\star w)_{\rho_0}\\
&=-c_2w_{\rho_0}(p\star w)+\Big((c_1+c_2-b-1)(p\star w)-c_1w\Big)(p\star w)_{\rho_0},
\endaligned
\ee
with the initial data
\bel{E4-3R1}
\aligned
w(0,\rho_0):=w_0(\rho_0)&=\overline{v}_0(\rho_0)-\alpha^2\overline{v}_{\rho_0\rho_0}(0,\rho_0)\\
&=u_0(x)-\alpha^2 u''_0(x)+{1\over b+1}({x\over T}+c_0),
\endaligned
\ee
and the boundary condition
\bel{E4-3R4}
w(\tau,\rho_0)|_{\rho_0=\pm\infty}=0,\quad \quad w_{\rho_0}(\tau,\rho_0)|_{\rho_0=\pm\infty}=0,
\ee
where we use 
$$
(p\star w)_{\rho_0\rho_0}=\alpha^{-2}(p\star w-w).
$$

It is easy to see there are two important terms of equation (\ref{E4-3}) as follows
$$
(1-{c_1\over b+1})w-{c_2(c_0+e^{\tau}\rho_0)\over b+1}w_{\rho_0},
$$
which determines the dissipative property of equation (\ref{E4-3}).

\subsection{A priori estimate}

We introduce a commutator estimate which can be found in \cite{Ka}.

\begin{lemma}
Let $s>0$. Then there is
\bel{Er4-1}
\|[\Lambda^s,u]v\|_{\LL^2}\leq C\Big(\|\del_xu\|_{\LL^{\infty}}\|\Lambda^{s-1}v\|_{\LL^2}+\|\Lambda^s u\|_{\LL^2}\|v\|_{\LL^{\infty}}\Big),
\ee
where positive constant $C$ depending on $s$.
\end{lemma}

We now derive a priori estimate of the solution for equation (\ref{E4-3}). 
Let $s>0$. Applying $\Lambda^s$ to both sides of (\ref{E4-3}), there is
\bel{E4-4}
\aligned
(\Lambda^sw)_{\tau}+(1-{c_1\over b+1})\Lambda^sw&-
e^{-\tau}\Lambda^s\Big[\Big({\Gamma\over\alpha^2}+{c_2(c_0+e^{\tau}\rho_0)\over b+1}\Big)w_{\rho_0}\Big]-(1-{c_1\over b+1})\Lambda^s(p\star w)\\
&\quad\quad+e^{-\tau}\Lambda^s\Big[\Big({\Gamma\over\alpha^2}+{c_0c_2+(c_2-b-1)e^{\tau}\rho)\over b+1}\Big)(p\star w)_{\rho_0}\Big]\\
&=-c_2\Lambda^s\Big(w_{\rho_0}(p\star w)\Big)+\Lambda^s\Big[\Big((c_1+c_2-b-1)(p\star w)-c_1w\Big)(p\star w)_{\rho_0}\Big].
\endaligned
\ee

\begin{lemma}
Let $s>2$ and $\alpha\neq0$. Assume that parameters $c_1,c_2,b$ satisfy
\bel{E4-4r0}
\aligned
&{2(b+1)+c_2-2c_1\over b+1}>0,\\
&{2c_1+1-c_2\over b+1}>0.
\endaligned
\ee
Then the solution $w$ of equation (\ref{E4-3}) satisfies 
$$
\|w\|_{\HH^s}\lesssim e^{-{2(b+1)+c_2-2c_1\over b+1}\tau}\|w_0\|_{\HH^s}.
$$
\end{lemma}
\begin{proof}
Taking the $\LL^2$-inner product with equation (\ref{E4-4}) by $\Lambda^sw$, we get
\bel{E4-5}
\aligned
&{1\over 2}{d\over d\tau}\|w\|^2_{\HH^s}+(1-{c_1\over b+1})\|w\|^2_{\HH^s}-e^{-\tau}\int_{\RR}\Lambda^sw\Lambda^s\Big[\Big({\Gamma\over\alpha^2}+{c_2(c_0+e^{\tau}\rho_0)\over b+1}\Big)w_{\rho_0}\Big]d\rho_0\\
&\quad\quad-(1-{c_1\over b+1})\int_{\RR}\Lambda^sw\Lambda^s(p\star w)d\rho_0
+e^{-\tau}\int_{\RR}\Lambda^s w\Lambda^s\Big[\Big({\Gamma\over\alpha^2}+{c_0c_2+(c_2-b-1)e^{\tau}\rho)\over b+1}\Big)(p\star w)_{\rho_0}\Big]d\rho_0\\
&\quad=\int_{\RR}\Lambda^sw\Lambda^s\Big[\Big((c_1+c_2-b-1)(p\star w)-c_1w\Big)(p\star w)_{\rho_0}\Big]d\rho_0-c_2\int_{\RR}\Lambda^sw\Lambda^s\Big(w_{\rho_0}(p\star w)\Big)d\rho_0.
\endaligned
\ee
Next we estimate each of terms in (\ref{E4-5}). On one hand, we use integration by parts to derive
\bel{E4-6R1}
\aligned
&\int_{\RR}\Lambda^sw\Lambda^s\Big[\Big({\Gamma\over\alpha^2}+{c_2(c_0+e^{\tau}\rho_0)\over b+1}\Big)w_{\rho_0}\Big]d\rho_0\\
&=\int_{\RR}\Big[\Big({\Gamma\over\alpha^2}+{c_2(c_0+e^{\tau}\rho_0)\over b+1}\Big)w_{\rho_0}\Big]\Lambda^{2s}wd\rho_0\\
&=-{c_2\over b+1}e^{\tau}\int_{\RR}\Lambda^sw\Lambda^swd\rho_0-{1\over2}\int_{\RR}\Big({\Gamma\over\alpha^2}+{c_2(c_0+e^{\tau}\rho_0)\over b+1}\Big)(\Lambda^{s}w)^2_{\rho_0}d\rho_0\\
&=-{c_2\over2(b+1)}e^{\tau}\|w\|_{\HH^s}^2,\\
&and\\
&(1-{c_1\over b+1})\int_{\RR}\Lambda^sw\Lambda^s(p\star w)d\rho_0=(1-{c_1\over b+1})\|w\|_{\HH^{s-1}}^2,\\
&and\\
&\int_{\RR}\Lambda^s w\Lambda^s\Big[\Big({\Gamma\over\alpha^2}+{c_0c_2+(c_2-b-1)e^{\tau}\rho)\over b+1}\Big)(p\star w)_{\rho_0}\Big]d\rho_0\\
&={b+1-c_2\over b+1}e^{\tau}\int_{\RR}\Lambda^{s-1}w\Lambda^{s-1}wd\rho_0
+{1\over2}\int_{\RR}\Big({\Gamma\over\alpha^2}+{c_0c_2+(c_2-b-1)e^{\tau}\rho)\over b+1}\Big)(\Lambda^{s-1}w)_{\rho_0}^2d\rho_0\\
&={3(b+1-c_2)\over2(b+1)}e^{\tau}\|w\|_{\HH^{s-1}}^2.
\endaligned
\ee

On the other hand, by (\ref{Er4-1}), H\"{o}lder inequality and $\HH^{s-1}\subset \LL^{\infty}$ with $s>2$, we use integration by parts to get
\bel{E4-6}
\aligned
\int_{\RR}\Lambda^sw\Lambda^s\Big(w_{\rho_0}(p\star w)\Big)d\rho_0&\lesssim\|w_{\rho_0}\|_{\LL^{\infty}}\int_{\RR}\Lambda^sw\Lambda^s(p\star w)d\rho_0\\
&\lesssim \|w\|_{\HH^{s}}^3,
\endaligned
\ee
and
\bel{E4-9}
\aligned
&\Big|\int_{\RR}\Lambda^sw\Lambda^s\Big[\Big((c_1+c_2-b-1)(p\star w)-c_1w\Big)(p\star w)_{\rho_0}\Big]d\rho_0\Big|\\
&=\Big|\int_{\RR}[\Lambda^s,(c_1+c_2-b-1)(p\star w)-c_1w](p\star w)_{\rho_0}\Lambda^{s} wd\rho_0\Big|\\
&\quad+\Big|\int_{\RR}\Big((c_1+c_2-b-1)(p\star w)-c_1w\Big)\Lambda^s(p\star w)_{\rho_0}\Lambda^{s} wd\rho_0\Big|\\
&\lesssim \Big[\Big((c_1+c_2-b-1)\|(p\star w)_{\rho_0}\|_{\LL^{\infty}}+c_1\|w_{\rho_0}\|_{\LL^{\infty}}\Big)\|\Lambda^{s-1}(p\star w)_{\rho_0}\|_{\LL^2}\\
&\quad+\Big((c_1+c_2-b-1)\|\Lambda^s(p\star w)\|_{\LL^2}+c_1\|\Lambda^s w\|_{\LL^2}\Big)\|(p\star w)_{\rho_0}\|_{\LL^{\infty}}\Big]\|w\|_{\HH^s}\\
&\quad +\Big(c_1+c_2-b-1)\|p\star w\|_{\LL^{\infty}}+c_1\|w\|_{\LL^{\infty}}\Big)\|w\|_{\HH^{s}}^2\\
&\lesssim \|w\|^3_{\HH^s}.
\endaligned
\ee

Thus using (\ref{E4-6R1})-(\ref{E4-9}), it follows from (\ref{E4-5}) that
$$
{d\over d\tau}\|w\|^2_{\HH^s}+{2(b+1)+c_2-2c_1\over b+1}\|w\|^2_{\HH^s}+{2c_1+1-c_2\over b+1}\|w\|^2_{\HH^{s-1}} \lesssim\|w\|^3_{\HH^s},
$$
then by (\ref{E4-4r0}), above inequality gives that
$$
{d\over d\tau}\|w\|^2_{\HH^s}+{2(b+1)+c_2-2c_1\over b+1}\|w\|^2_{\HH^s}\lesssim\|w\|^3_{\HH^s},
$$
which is a Bernoulli-type differential inequality, it is equivalent to 
$$
-{d\over d\tau}\|w\|^{-1}_{\HH^s}+{2(b+1)+c_2-2c_1\over b+1}\|w\|^{-1}_{\HH^s}\lesssim 1,
$$
this means that
$$
\|w\|_{\HH^s}\lesssim e^{-{2(b+1)+c_2-2c_1\over b+1}\tau}\|w_0\|_{\HH^s}.
$$
\end{proof}

\subsection{Nonlinear stability of self-similar blowup solutions}

Since we study the asymptotic stability of explicit self-similar solutions for the generalized $b$-equation (\ref{E1-1R00}), it is equivalent to prove global-well posedness for equation (\ref{E4-3}) with the initial data (\ref{E4-3R1}) and boundary condition (\ref{E4-3R4}). 

Let the linear operator $\Lcal$ be the form
\bel{E4-11}
\aligned
\Lcal [w]:&=-(1-{c_1\over b+1})w+e^{-\tau}\Big({\Gamma\over\alpha^2}+{c_2(c_0+e^{\tau}\rho_0)\over b+1}\Big)w_{\rho_0}+(1-{c_1\over b+1})(p\star w)\\
&\quad\quad-e^{-\tau}\Big({\Gamma\over\alpha^2}+{c_0c_2+(c_2-b-1)e^{\tau}\rho)\over b+1}\Big)(p\star w)_{\rho_0}.
\endaligned.
\ee

Equation (\ref{E4-3}) can be rewritten as
\bel{E4-12}
w_t=\Lcal [w]+f(w),
\ee
where the nonlinear term
\bel{E4-12R1}
f(w):=\Big((c_1+c_2-b-1)(p\star w)-c_1w\Big)(p\star w)_{\rho_0}-c_2w_{\rho_0}(p\star w).
\ee

\begin{lemma}
Let $s>2$. There is $\Lcal [w]\in\HH^s$, for any $w\in\Dcal{(\Lcal)}$.
\end{lemma}
\begin{proof}
Since there is no singular coefficient in the linear operator $\Lcal$ and the highest order derivative on $\rho_0$ is $1$,
It follows from (\ref{E4-11}) that this result holds for $\rho_0\in\RR$.
\end{proof}

\begin{lemma}
Let $s>2$. The linear operator $\Lcal$ defined in (\ref{E4-11}) is a closed and densely defined linear dissipative operator in $\HH^s$.
\end{lemma}
\begin{proof}
It is easy to check that $\Lcal [w]$ is a densely defined and closed linear operator in $\HH^s$. Here,we only prove that $\Lcal$ is dissipative, i.e. 
$$(\Lcal [w],w)_s\leq0.$$

To see this, direct computations give that
\bel{E4-13}
\aligned
&\int_{\Omega}(\Lambda^s\Lcal [w])\Lambda^swd\rho_0\\
&=-(1-{c_1\over b+1})\|w\|^2_{\HH^s}+e^{-\tau}\int_{\RR}\Lambda^sw\Lambda^s\Big[\Big({\Gamma\over\alpha^2}+{c_2(c_0+e^{\tau}\rho_0)\over b+1}\Big)w_{\rho_0}\Big]d\rho_0\\
&\quad\quad+(1-{c_1\over b+1})\int_{\RR}\Lambda^sw\Lambda^s(p\star w)d\rho_0
-e^{-\tau}\int_{\RR}\Lambda^s w\Lambda^s\Big[\Big({\Gamma\over\alpha^2}+{c_0c_2+(c_2-b-1)e^{\tau}\rho)\over b+1}\Big)(p\star w)_{\rho_0}\Big]d\rho_0
\endaligned
\ee
which combining with (\ref{E4-6R1}) and (\ref{E4-4r0}) means that
$$
\int_{\Omega}(\Lambda^s\Lcal [w])\Lambda^swd\rho_0=-{2(b+1)+c_2-2c_1\over b+1}\|w\|^2_{\HH^s}-{2c_1+1-c_2\over b+1}\|w\|^2_{\HH^{s-1}}<0.
$$
This completes the proof.
\end{proof}

\begin{lemma}
Let $s>2$. The operator $\Lcal$ defined in (\ref{E4-11}) is invertible in $\HH^s$. Moreover, the operator $\Lcal$ generates a $\mathbb{C}_0$-semigroup $(\textbf{S}(t))_{\tau\geq0}$ in $\HH^s$.
\end{lemma}
\begin{proof}
To see the existence of $\Lcal^{-1}$, we need to prove the operator $\Lcal$ are injective and surjective.
We first show $\Lcal$ is injective. Let $w\in\Dcal{(\Lcal)}$ such that
$$
\Lcal [w]=0,
$$
which gives that
\bel{E4-14}
\aligned
-(1-{c_1\over b+1})\Lambda^sw&+
e^{-\tau}\Lambda^s\Big[\Big({\Gamma\over\alpha^2}+{c_2(c_0+e^{\tau}\rho_0)\over b+1}\Big)w_{\rho_0}\Big]+(1-{c_1\over b+1})\Lambda^s(p\star w)\\
&\quad\quad-e^{-\tau}\Lambda^s\Big[\Big({\Gamma\over\alpha^2}+{c_0c_2+(c_2-b-1)e^{\tau}\rho)\over b+1}\Big)(p\star w)_{\rho_0}\Big]=0.
\endaligned
\ee

Multiplying (\ref{E4-14}) by $\Lambda^sw$, and integrating by parts over $\RR$, we derive
$$
\aligned
\int_{\RR}\Lambda^s\Lcal [w]\Lambda^swd\rho_0=-{2(b+1)+c_2-2c_1\over b+1}\|w\|^2_{\HH^s}-{2c_1+1-c_2\over b+1}\|w\|^2_{\HH^{s-1}}=0,
\endaligned.
$$
which combining with the boundary condition (\ref{E4-3R4}) implies that $w=0$. So the operator $\Lcal$ is injective.

Next, we show the operator $\Lcal$ is surjective. $\forall g\in\HH^1$, set
\bel{E4-15}
\Lcal [w]=g.
\ee

Applying $\Lambda^s$ to equation (\ref{E4-15}), then multiplying it by $\Lambda^sw$, and integrating by parts over $\Omega$, 
$$
{2(b+1)+c_2-2c_1\over b+1}\|w\|^2_{\HH^s}+{2c_1+1-c_2\over b+1}\|w\|^2_{\HH^{s-1}}=-2\int_{\Omega}g\Lambda^swd\rho_0,
$$
from which, using Young's inequality, we have
$$
\|w\|_{\HH^s}\leq C\|g\|_{\HH^s}.
$$
By the standard theory of elliptic-type equations of the general order, there exists a unique weak solution $w\in\HH^1$. For such a solution, we have $w\in\HH^{s+1}$ if for $g\in\HH^s$.
So the operator $\Lcal$ is surjective. Furthermore, by the Lumer-Phillips Theorem \cite{P}, the operator $\Lcal$ generates a $\mathbb{C}_0$-semigroup $(\textbf{S}(t))_{\tau\geq0}$ in $\HH^s$.
\end{proof}

 By Lemma 3.3-3.5, we can conclude the following result.
\begin{proposition}
Let $s>2$. The operator $\Lcal$ defined in (\ref{E4-11}) generates a $\mathbb{C}_0$-semigroup $(\textbf{S}(t))_{\tau\geq0}$ in $\HH^s$.
Moreover, the Cauchy problem 
\begin{eqnarray*}
&&\frac{d}{d \tau}w(\tau)=\Lcal w(\tau),\\
&&w(0)=w_0,
\end{eqnarray*}
with the vanishing boundary has a unique solution
\begin{equation*}
w(\tau)=\textbf{S}(\tau)w_0,
\end{equation*}
where the initial data $w_0$ is given in (\ref{E4-3R1}).
\end{proposition}

We now return to nonlinear equation (\ref{E4-12}). Using Duhamel's formula and Proposition 3.1, equation (\ref{E4-12}) can be formulated as an abstract integral equation
$$
w(\tau)=\textbf{S}(\tau)w_0+\int_0^{\tau}\textbf{S}(\tau-s)f(w(s))ds.
$$

Let $s>2$ be a fixed constant. Define a closed ball in $\HH^s$ with radius $\sigma<1$ as follows
$$
\textbf{B}_{\sigma}:=\{w\in\HH^s |\quad \|w\|_{\HH^s}<\sigma\},
$$
and the solution map $\Tcal$ as follows
\bel{E4-16}
\Tcal w(\tau):=\textbf{S}(\tau)w_0+\int_0^t\textbf{S}(\tau-s)f(w(s))ds.
\ee
In what follows, we should prove that $\Tcal w(\tau)=w(\tau)$ by employing the Banach fixed point theorem.

To apply Banach fixed point theorem, we need to use the following inequality for the weighted Sobolev space.

\begin{lemma}\cite{Ka}
Let $s>0$. Then $\HH^s\cap\LL^{\infty}$ is an algebra, and
$$
\|uv\|_{\HH^s}\leq C\Big(\|u\|_{\LL^{\infty}}\|v\|_{\HH^s}+\|u\|_{\HH^s}\|v\|_{\LL^{\infty}}\Big),
$$
where $C$ is a postive constant depending on $s$.
\end{lemma}

\begin{lemma}
Let $s>2$ be a fixed constant. Assume that $\|w_0\|_{\HH^{s+1}}<\sigma$ with $0<\sigma\ll1$. The map $\Tcal$ defined in (\ref{E4-16}) takes $\textbf{B}_{\sigma}$ into itself.
\end{lemma}
\begin{proof}
By (\ref{E4-12R1}) and Lemma 3.6, we derive
$$
\aligned
\|f(w)\|_{\HH^{s}}
&\leq (c_1+c_2-b-1)\|(p\star w)(p\star w)_{\rho_0}\|_{\HH^s}+c_1\|(p\star w)_{\rho_0}w\|_{\HH^s}+\|(p\star w)w_{\rho_0}\|_{\HH^s}\\
&\leq C_{c_1,c_2,b}\Big(\|(p\star w)_{\rho_0}\|_{\LL^{\infty}}(\|p\star w\|_{\HH^{s}}+\|w\|_{\HH^{s}})+\|w_{\rho_0}\|_{\LL^{\infty}}\|p\star w\|_{\HH^{s}}\Big),
\endaligned
$$
then using $\HH^{s}\subset\LL^{\infty}$, $w=\Lambda^s(p(\rho_0)\star \overline{v})$ and Lemma 3.2, above inequality gives that
$$
\aligned
\|f(w)\|_{\HH^{s}}&\leq C_{c_1,c_2,b}\|w\|^2_{\HH^{s}}<C_{c_1,c_2,b}\sigma^2<\sigma.
\endaligned
$$
where $C_{c_1,c_2,b}$ is a positive constant depending on parameters $c_1,c_2,b$.

Hence we conclude that the map $\Tcal$ defined in (\ref{E4-16}) takes $\textbf{B}_{\sigma}$ into itself.
\end{proof}

\begin{lemma}
Let $s>2$ be a fixed constant. Assume that $\|w_0\|_{\HH^{s+1}}<\sigma$ with $0<\sigma\ll1$. The nonlinear equation
(\ref{E4-12}) with the initial data (\ref{E4-12R1})and the boundary condition (\ref{E4-3R4}) has a unique solution $w\in\textbf{B}_{\sigma}$.
\end{lemma}
\begin{proof}
It is equivalent to prove that the solution map $\Tcal$ given in (\ref{E4-16}) has a fixed point in $\textbf{B}_{\sigma}$.
For any two $w,\overline{w}$ in $\textbf{B}_{\sigma}$, by (\ref{E4-12R1}) and Lemma 3.6, direct computations give that
\bel{E4-17}
\aligned
\|f(w)-f(\overline{w})\|_{\HH^{s}}&\lesssim\|(p\star w)(p\star w)_{\rho_0}-(p\star \overline{w})(p\star \overline{w})_{\rho_0}\|_{\HH^s}
+\|(p\star w)_{\rho_0}w-(p\star \overline{w})_{\rho_0}\overline{w}\|_{\HH^s}\\
&\quad\quad+\|(p\star w)w_{\rho_0}-(p\star \overline{w})\overline{w}_{\rho_0}\|_{\HH^s}\\
&\lesssim \|(p\star (w-\overline{w}))(p\star w)_{\rho_0}\|_{\HH^s}+\|(p\star \overline{w})(p\star (w-\overline{w}))_{\rho_0}\|_{\HH^s}\\
&\quad+\|((p\star (w-\overline{w}))_{\rho_0}w\|_{\HH^s}+\|(p\star \overline{w})_{\rho_0}(w-\overline{w})\|_{\HH^s}\\
&\quad+\|(p\star(w-\overline{w}))w_{\rho_0}\|_{\HH^s}+\|(p\star\overline{w})(w-\overline{w})_{\rho_0}\|_{\HH^s}\\
&\leq C_{c_1,c_2,b}\sigma\|w-\overline{w}\|_{\HH^s}.
\endaligned
\ee

Hence, it follows from (\ref{E4-16}) and (\ref{E4-17}) that
$$
\|\Tcal w(\tau)-\Tcal\overline{w}(\tau)\|_{\HH^s}\lesssim \sigma\|w-\overline{w}\|_{\HH^s},
$$
which implies that the solution map $\Tcal$ is contracting when we choose a sufficient small positive constant $\sigma$. 
Thus using the Banach fixed point theorem, the map $\Tcal$ defined in (\ref{E4-16}) has a fixed point in $\textbf{B}_{\sigma}$. The fixed point is the solution of nonlinear equation (\ref{E4-12}).

\end{proof}

We now return to the existence of solution for nonlinear equation (\ref{E4-2}).

\begin{proposition}
Let $s>3$. The nonlinear equation (\ref{E4-2}) with the initial data (\ref{E4-3R1}) and boundary condition (\ref{E4-3R4}) admits a global solution $v(\tau,\rho)\in\HH^s$. Moreover, if the initial data $\|v_0\|_{\HH^{s+1}}<\sigma$ with $0<\sigma\ll1$, then there is
$$
\|v\|_{\HH^s}\leq C\sigma e^{-2\tau}.
$$
\end{proposition}
\begin{proof}
By Lemma 3.8, we have a global solution of equation (\ref{E4-2}) as follows
\bel{E4-Y1}
v(\tau,\rho)=e^{\tau}\overline{v}(\tau,\rho_0)=e^{\tau}(p\star w(\tau,\rho_0)),
\ee
where $w(\tau,\rho_0)$ is a global solution of equation (\ref{E4-12}) given in Lemma 3.8, and $\rho_0=e^{-\tau}\rho$.

Furthermore, It follows from (\ref{E4-Y1}) that $v_{\rho\rho}=e^{-\tau}w$. So by Lemma 3.2, we derive
$$
\|v\|_{\HH^s}\leq e^{-\tau}\|w\|_{\HH^{s-2}}\lesssim e^{-{3(b+1)+c_2-2c_1\over b+1}\tau}\|w_0\|_{\HH^{s-2}}\lesssim \sigma e^{-{3(b+1)+c_2-2c_1\over b+1}\tau}.
$$

\end{proof}

From Proposition 3.2, it is easy to see nonlinear stability of explicit self-similar blowup solution (\ref{ER1-01}) for the generalized $b$-equation (\ref{E1-1R00})
 by asssumption (\ref{E4-4r0}).

\subsection{Instability of self-similar blowup solutions}

It is a natural problem to consider the case of parameters unsatisfying the asssumption (\ref{E4-4r0}).
This is related to instability of self-similar blowup solutions for the generalized $b$-equation (\ref{E1-1R00}).

\begin{proposition}
Let $s>2$ and $\alpha\neq0$. Assume that parameters $c_1,c_2,b$ satisfy
\bel{E5-1}
{2(b+1)+c_2-2c_1\over b+1}<0,
\ee
or
\bel{E5-1R0}
\aligned
&2(b+1)+c_2-2c_1=0,\\
&{2c_1+1-c_2\over b+1}< 0.
\endaligned
\ee
Then the solution $w$ of equation (\ref{E4-3}) satisfies 
$$
\|w\|_{\HH^s}\geq e^{C_{c_1,c_2,b}\tau}\|w_0\|_{\HH^s},
$$
where $C_{c_1,c_2,b}$ is a positive constant depending on $c_1,c_2,b$.
\end{proposition}
\begin{proof}
We directly apply (\ref{E4-6R1})-(\ref{E4-9}) to (\ref{E4-5}), there is 
\bel{E5-2}
{d\over d\tau}\|w\|^2_{\HH^s}+{2(b+1)+c_2-2c_1\over b+1}\|w\|^2_{\HH^s}+{2c_1+1-c_2\over b+1}\|w\|^2_{\HH^{s-1}} \lesssim\|w\|^3_{\HH^s}.
\ee

If (\ref{E5-1}) holds,
then by (\ref{E5-2}), there is a positive constant $C_{c_1,c_2,b}$ depending on $c_1,c_2,b$ such that
$$
{d\over d\tau}\|w\|^2_{\HH^s}-C_{c_1,c_2,b}\|w\|^2_{\HH^s}\lesssim\|w\|^3_{\HH^s},
$$
which is a Bernoulli-type differential inequality, it is equivalent to 
$$
-{d\over d\tau}\|w\|^{-1}_{\HH^s}-C_{c_1,c_2,b}\|w\|^{-1}_{\HH^s}\lesssim 1,
$$
this means that
\bel{E5-4}
\|w\|_{\HH^s}\geq e^{C_{c_1,c_2,b}\tau}\|w_0\|_{\HH^s}.
\ee

If (\ref{E5-1R0}) holds, then by (\ref{E5-2}) and $\HH^{s}\subset\HH^{s-1}$ with $s>1$, there is a positive constant $C'_{c_1,c_2,b}$ depending on $c_1,c_2,b$ such that
$$
{d\over d\tau}\|w\|^2_{\HH^s}-C'_{c_1,c_2,b}\|w\|^2_{\HH^{s}}\lesssim\|w\|^3_{\HH^s},
$$
which also gives (\ref{E5-4}).
\end{proof}

It follows from Proposition 3.3 that the instability of self-similar blowup solutions holds for the generalized $b$-equation (\ref{E1-1R00}) by the assumption (\ref{E5-1}) or (\ref{E5-1R0}).

\


\
\
\

\textbf{Acknowledgments.} 
The author is supported by NSFC No 11771359.

\end{document}